\documentclass{article}
\usepackage{arxiv}
\usepackage{amssymb}
\usepackage[utf8]{inputenc} 
\usepackage[T1]{fontenc}    
\usepackage{hyperref}       
\usepackage{url}            
\usepackage{booktabs}       
\usepackage{amsfonts}       
\usepackage{nicefrac}       
\usepackage{microtype}      
\usepackage{graphicx}
\usepackage{natbib}
\usepackage{doi}
\usepackage{amsmath}
\usepackage{physics}
\usepackage{wrapfig}
\usepackage{bm}
\usepackage{mathrsfs}

\usepackage[linesnumbered,ruled]{algorithm2e}
\usepackage[noend]{algpseudocode}
\usepackage{xcolor}
\usepackage{forest}
\usepackage{tcolorbox}
\usepackage{tikz-cd}
\usepackage{enumitem}

\newenvironment{proof}{\paragraph{Proof:}}{\hfill$\square$}

\def\Xint#1{\mathchoice
{\XXint\displaystyle\textstyle{#1}}%
{\XXint\textstyle\scriptstyle{#1}}%
{\XXint\scriptstyle\scriptscriptstyle{#1}}%
{\XXint\scriptscriptstyle\scriptscriptstyle{#1}}%
\!\int}
\def\XXint#1#2#3{{\setbox0=\hbox{$#1{#2#3}{\int}$ }
\vcenter{\hbox{$#2#3$ }}\kern-.6\wd0}}

\def\fint{\Xint-}

\newtheorem{theorem}{Theorem}
\newtheorem{proposition}[theorem]{Proposition}
\newtheorem{lemma}[theorem]{Lemma}

\newtheorem{assumption}[theorem]{Assumption}
\newtheorem{remark}[theorem]{Remark}

\title{Generalised higher order vectorial $\infty$-eigenvalue problems}
\date{} 					

\author{ \href{https://orcid.org/0000-0003-1700-756X}{William Chang} \\
	Department of Applied Mathematics\\
	University of California, Los Angeles\\
	Los Angeles, CA \\
	USA\\
	\url{https://williamc.me/} \\
	\And
	\href{http://orcid.org/0000-0001-5292-270X}{Nikos Katzourakis} \\
	Department of Mathematics and Statistics\\
	University of Reading, Reading\\
	Whiteknights Campus\\
          Pepper Lane, RG6 6AX\\
	 UK \\
}

\renewcommand{\headeright}{}
\renewcommand{\undertitle}{}

\hypersetup{
pdftitle={Generalised higher order vectorial infinity-eigenvalue problems},
pdfsubject={Calculus of Variations in L-infinity},
pdfauthor={William Chang},
pdfkeywords={Calculus of Variations in L-infinity, infinity-eigenvalue problem, higher-order problems, Lagrange multipliers},
}

\begin{document}
\maketitle
\begin{abstract}
We study the problem of minimising the $L^\infty$ norm of a function of the $k$-th derivative over a class of maps, subject to a constraint involving the $L^\infty$ norm of a function of the map and its lower-order derivatives, for any integer $k\ge 2$. We impose boundary conditions corresponding to the $k$-th order analogues of the classical ``clamped'' and ``hinged'' cases. By employing the method of $L^p$ approximations, we establish the existence of a special $L^\infty$ minimiser, which solves a divergence PDE system with measure coefficients as parameters. This system constitutes the counterpart of the Aronsson--Euler equations for the constrained variational problem under consideration. Furthermore, we establish a lower bound for the eigenvalue. The present work extends the second-order vectorial results of Clark and Katzourakis \cite{clark2024generalized} to the general higher-order setting.
\end{abstract}

\medskip

\thanks{\it N.K.\ has been partially financially supported through the EPSRC grant EP/X017109/1.}

\medskip



\section{Introduction}

Let $n,N\in\mathbb{N}$ with $n\ge2$, and let $\Omega\subseteq \mathbb{R}^n$ be a bounded domain with Lipschitz boundary $\partial\Omega$. In this work we study nonlinear higher-order eigenvalue problems arising in the calculus of variations in $L^\infty$. Our aim is to extend recent results on second-order vectorial $\infty$-eigenvalue problems obtained in \cite{clark2024generalized} to arbitrary order.

More precisely, for an integer $k\ge2$ we consider the constrained supremal minimisation problem
\begin{equation}\label{minimisation problem}
\|f(D^k u_\infty)\|_{L^\infty(\Omega)}
=
\inf \Big\{
\|f(D^k v)\|_{L^\infty(\Omega)} :
v \in W_B^{k,\infty}(\Omega;\mathbb{R}^N),
\ \|g(v,Dv,\ldots,D^{k-1}v)\|_{L^\infty(\Omega)} = 1
\Big\}.
\end{equation}

Our objective is to derive a necessary optimality condition for minimisers of \eqref{minimisation problem}. As is typical in supremal variational problems, the resulting condition takes the form of a nonlinear PDE system satisfied by the minimiser. However, as this is not a problem involving conventional integral functionals, the situation is considerably different. 

The main difficulty is that functionals defined through the $L^\infty$ norm are not G\^ateaux differentiable. Consequently, classical variational arguments leading to Euler--Lagrange equations cannot be applied directly. Instead, we analyse a family of auxiliary $L^p$ problems and pass to the limit as $p\to\infty$. Even though it would have been possible to solve the constrained $L^\infty$ minimisation problem \eqref{minimisation problem} directly, the passage through $L^p$ approximations is an essential technical device in order to bypass the lack of regularity of supremal functionals and obtain the relevant extremality conditions.

The admissible class $W_B^{k,\infty}(\Omega;\mathbb{R}^N)$ appearing in \eqref{minimisation problem} depends on the boundary conditions imposed on the maps. We consider the two natural $k$-th order analogues of classical clamped and hinged boundary conditions:
\begin{align}
W_C^{k,\infty}(\Omega;\mathbb{R}^N)
&:= W_0^{k,\infty}(\Omega;\mathbb{R}^N),\\
W_H^{k,\infty}(\Omega;\mathbb{R}^N)
&:= W^{k,\infty}(\Omega;\mathbb{R}^N)\cap W_0^{k-1,\infty}(\Omega;\mathbb{R}^N).
\end{align}

Functions in the clamped space satisfy
\[
u = Du = D^2u = \cdots = D^{k-1}u = 0 \quad \text{on } \partial\Omega,
\]
while functions in the hinged space satisfy
\[
u = Du = \cdots = D^{k-2}u = 0 \quad \text{on } \partial\Omega,
\]
in the trace sense. These correspond to higher-order Dirichlet-type boundary conditions.

\paragraph{Background.}
Problem \eqref{minimisation problem} lies within the framework of the calculus of variations in $L^\infty$, a theory originating in the work of Aronsson in the 1960s \cite{aronsson1967extension,aronson1968non}. Since then the subject has developed substantially, with connections to fully nonlinear PDE and optimisation problems; see for instance \cite{aronsson2004tour,juutinen2001,katzourakis2015}.

Supremal variational problems exhibit several structural differences from classical integral functionals. In particular, the $L^\infty$ norm lacks differentiability properties required for standard Euler--Lagrange theory. Moreover, the space $L^\infty$ is neither reflexive nor separable, which complicates compactness and localisation arguments. These features make the analysis of such problems significantly more delicate.

Higher-order and constrained $L^\infty$ variational problems have received comparatively little attention. Existing works include \cite{dalMaso2004higher,katzourakis2013,barron2005,clark2021data,clark2024isosupremic}. The present paper can be viewed as a higher-order extension of the framework developed in \cite{clark2024isosupremic}, which itself generalised the existence and PDE characterisation results of \cite{clark2021data} for second-order systems. Those works in turn build upon earlier studies of scalar $\infty$-eigenvalue problems for the $\infty$-Laplacian \cite{juutinen1999,juutinen2005,juutinen2001}. Further developments and applications may be found in \cite{katzourakis2026minimisation,katzourakis2020minimisation}.

\paragraph{Approximation approach.}
Due to the vectorial and higher-order structure of \eqref{minimisation problem}, techniques based on viscosity solutions are not applicable; see \cite{katzourakis2015} for an overview of that theory. Instead we employ an approximation strategy based on $L^p$ variational problems.

The underlying idea is that, on sets of finite measure, the $L^p$ norm of a given $L^\infty$ function converges to its $L^\infty$ norm, as $p\to\infty$. By analysing minimisers of suitable $L^p$ problems and establishing uniform estimates, we obtain limiting objects that capture the structure of the supremal problem. For vectorial higher-order problems this approach is currently the main available method. Even the intrinsic duality theory of Bungert and Korolev \cite{bungert2022} applies only in the scalar first-order setting.

\paragraph{Structural assumptions.}
Throughout the paper we impose the following hypotheses on the supremand $f$ and the constraint function $g$ (for the notation used throughout this paper, we refer to Section \ref{sec:preliminary}).

\begin{assumption}\label{f}
The function \(f\) satisfies:
\begin{enumerate}
    \item \(f \in C^1\big(\mathbb{R}_s^{N\times n^k}\big)\).
    
    \item\label{morrey} \(f\) is (Morrey) \(k\)-quasiconvex, that is,
    \[
    f(X)
    \leq
    \fint_\Omega f\big(X+D^k\phi(x)\big)\,d\mathcal{L}^n(x),
    \]
    for every \(X\in \mathbb{R}_s^{N\times n^k}\) and every
    \(\phi\in W_0^{k,\infty}(\Omega;\mathbb{R}^N)\).

    \item There exist constants \(0<C_1\le C_2\) such that, for every
    \(X\in\mathbb{R}_s^{N\times n^k}\setminus\{0\}\),
    \begin{equation}\label{1.3c}
        0<C_1 f(X)\le \partial f(X):X \le C_2 f(X).
    \end{equation}

    \item There exist constants \(C_3,\ldots,C_6>0\), \(\alpha>1\), and \(\beta\le1\) such that, for every
    \(X\in\mathbb{R}_s^{N\times n^k}\),
    \begin{align}
        -C_3 + C_4|X|^\alpha
        &\le f(X)\le C_5|X|^\alpha + C_6,\label{growth-a}\\
        |\partial f(X)|
        &\le C_5 f(X)^\beta + C_6.\label{growth-b}
    \end{align}
\end{enumerate}
\end{assumption}

Next we introduce the assumptions on the constraint function \(g\).

\begin{assumption}\label{g}
The function \(g\) satisfies:
\begin{enumerate}
    \item
    \[
    g\in C^1\!\bigg(\mathbb{R}^N\times\prod_{j=1}^{k-1}\mathbb{R}^{N\times n^j}\bigg).
    \]

    \item\label{coercive}
    \(g\) is coercive, in the sense that
    \[
    \lim_{t\to\infty}
    \left(
    \inf_{\substack{(\eta,P_1,\ldots,P_{k-1})\in \mathbb{R}^N\times\prod_{j=1}^{k-1}\mathbb{R}^{N\times n^j}\\
    |(\eta,P_1,\ldots,P_{k-1})|=1}}
    g(t\eta,tP_1,\ldots,tP_{k-1})
    \right)
    =\infty.
    \]

    \item\label{partial}
    There exist constants \(0<C_7\le C_8\) such that, for every
    \[
    (\eta,P_1,\ldots,P_{k-1})\in
    \bigg(\mathbb{R}^N\times\prod_{j=1}^{k-1}\mathbb{R}^{N\times n^j}\bigg)\setminus\{0\},
    \]
    one has
    \[
    0<C_7 g(\eta,P_1,\ldots,P_{k-1})
    \le
    \partial_\eta g\cdot\eta
    +\sum_{j=1}^{k-1}\partial_{P_j}g:P_j
    \le
    C_8 g(\eta,P_1,\ldots,P_{k-1}).
    \]
\end{enumerate}
\end{assumption}

These conditions guarantee appropriate regularity, growth, and coercivity properties required for the analysis of the constrained supremal problem \eqref{minimisation problem}.

\section{Preliminaries}\label{sec:preliminary}

In this section we record the notation and structural assumptions used throughout the paper.

\paragraph{Derivative notation.}
Let \(u \in C^{k}(\Omega)\) be scalar-valued. For any integer \(m\) with \(1\le m\le k\), its \(m\)-th derivative at a point \(x\in\Omega\) is the symmetric \(m\)-th order tensor
\[
D^mu(x)=\big(D^m_{i_1\cdots i_m}u(x)\big)_{i_1,\ldots,i_m=1}^n \in \mathbb{R}^{n^m}_s,
\]
which acts as an \(m\)-linear form on \((\mathbb{R}^n)^m\) by
\[
D^mu(x)[v_1,\ldots,v_m]
:=
\sum_{i_1,\ldots,i_m=1}^n D^m_{i_1\cdots i_m}u(x)\,v_1^{i_1}\cdots v_m^{i_m}.
\]

Now let \(u=(u^1,\dots,u^N):\Omega\to\mathbb{R}^N\) be a vector-valued map of class \(C^k\). Then
\[
D^ku(x)\in \mathbb{R}_s^{N\times n^k},
\]
with components
\[
\big(D^ku(x)\big)^\ell_{i_1\cdots i_k}:=D^k_{i_1\cdots i_k}u^\ell(x),
\qquad
\ell=1,\dots,N,\quad i_1,\ldots,i_k=1,\dots,n.
\]
Equivalently, \(D^ku(x)\) may be regarded as a \(k\)-linear map
\[
D^ku(x):(\mathbb{R}^n)^k\to\mathbb{R}^N
\]
defined component-wise by
\[
\big(D^ku(x)[v_1,\ldots,v_k]\big)^\ell
=
\sum_{i_1,\ldots,i_k=1}^n D^k_{i_1\cdots i_k}u^\ell(x)\,v_1^{i_1}\cdots v_k^{i_k}.
\]

\paragraph{Inner products.}
For tensors of the same order, we write \(A:B\) for the Euclidean inner product. For example, if
\(A,B\in\mathbb{R}^{N\times n^m}\) for some integer \(m\ge1\), then \(A:B\) denotes the component-wise Euclidean contraction over all indices:
\[
A:B = \sum_{\ell=1}^N\sum_{i_1,\ldots,i_m=1}^n A^\ell_{i_1\cdots i_m} B^\ell_{i_1\cdots i_m}.
\]

\paragraph{Function spaces.}
We use the following notation:
\begin{align}
C_{\mathrm B}^k(\overline{\Omega};\mathbb{R}^N)
&:=
C^k(\overline{\Omega};\mathbb{R}^N)\cap W_{\mathrm B}^{k,\infty}(\Omega;\mathbb{R}^N),\\
W_{\mathrm C}^{k,p}(\Omega;\mathbb{R}^N)
&:=
W_0^{k,p}(\Omega;\mathbb{R}^N),
\qquad p\in[1,\infty),\\
W_{\mathrm H}^{k,p}(\Omega;\mathbb{R}^N)
&:=
W^{k,p}(\Omega;\mathbb{R}^N)\cap W_0^{k-1,p}(\Omega;\mathbb{R}^N),
\qquad p\in[1,\infty).
\end{align}
Here, as in the introduction, \(W_{\mathrm B}^{k,p}\) denotes either the clamped space \(W_{\mathrm C}^{k,p}\) or the hinged space \(W_{\mathrm H}^{k,p}\).

We may now state the main result.

\begin{theorem}[Main result]\label{main theorem}
Let \(k\ge 2\) be an integer. Assume that \ref{f} and \ref{g} hold. Then the following conclusions are valid.

\medskip
\noindent\textnormal{(A) Existence of a minimiser.}
The constrained \(L^\infty\) minimisation problem \eqref{minimisation problem} admits a solution
\[
u_\infty \in W_{\mathrm B}^{k,\infty}(\Omega;\mathbb{R}^N).
\]

\medskip
\noindent\textnormal{(B) Limiting Euler--Lagrange structure.}
There exist Radon measures
\[
\mathrm{M}_\infty \in \mathcal{M}\big(\overline{\Omega};\mathbb{R}_s^{N\times n^k}\big),
\qquad
\nu_\infty \in \mathcal{M}(\overline{\Omega}),
\]
such that
\begin{equation}\label{limiting PDE}
\int_{\overline{\Omega}} D^k\phi : d\mathrm{M}_\infty
=
\Lambda_\infty
\int_{\overline{\Omega}}
\bigg(
\partial_\eta g_\infty\cdot \phi
+
\sum_{j=1}^{k-1}\partial_{P_j} g_\infty:D^j\phi
\bigg)\,d\nu_\infty,
\end{equation}
for every test map \(\phi\in C_{\mathrm B}^k(\overline{\Omega};\mathbb{R}^N)\). For the sake of brevity, in the above we have symbolised
\[
g_\infty \equiv  g(u_\infty,Du_\infty,\ldots,D^{k-1}u_\infty),
\]
and similarly for its partial derivatives, and
\[
\Lambda_\infty := \|f(D^ku_\infty)\|_{L^\infty(\Omega)} > 0
\]
is the associated \(L^\infty\)-eigenvalue.

\medskip
Moreover, \(\Lambda_\infty\) satisfies the a priori lower bound
\[
\Lambda_\infty \geq \left( C_4  \Bigg\{\displaystyle\sum_{j=0}^{k-1}
 \bigg( \prod_{l=j}^{k-1} C_{l}(\Omega) \bigg)\|\partial_{P_j} g\|_{L^\infty(\{g\le1\})} \Bigg\}^{-\alpha}
-C_3
\right)^+,
\]
where we have used the notation \(\partial_{P_0}g \equiv \partial_\eta g\), and \((\cdot)^+\) denotes the positive part. 

In the above estimate, \(C_0(\Omega),\ldots,C_{k-2}(\Omega)\) denote constants arising from the associated Poincar\'e inequality applied to the derivatives of corresponding order, and $C_{k-1}(\Omega)$ denotes the constant arising from either the Poincar\'e, or the Poincar\'e-Wirtinger inequality applied to the derivatives of $(k-1)$-order, depending on the type of boundary data (clamped/hinged).

\medskip
\noindent\textnormal{(C) Approximation by \(L^p\) problems.}
There exists a sequence \((p_j)_{j\ge1}\subseteq (n/\alpha,\infty)\) with \(p_j\to\infty\) such that, for each \(p=p_j\), there is a quadruple
\[
(u_p,\Lambda_p,\mathrm{M}_p,\nu_p)
\in
W_{\mathrm B}^{k,\alpha p}(\Omega;\mathbb{R}^N)
\times [0,\infty)
\times \mathcal{M}\big(\overline{\Omega};\mathbb{R}_s^{N\times n^k}\big)
\times \mathcal{M}(\overline{\Omega})
\]
satisfying
\[
\begin{cases}
u_p \to u_\infty & \text{in } C^{k-1}(\overline{\Omega};\mathbb{R}^N),\\
D^ku_p \rightharpoonup D^ku_\infty & \text{in } L^q(\Omega;\mathbb{R}_s^{N\times n^k}) \quad \forall q\in(1,\infty),\\
\Lambda_p \to \Lambda_\infty,\\
\mathrm{M}_p \stackrel{*}{\rightharpoonup} \mathrm{M}_\infty,\\
\nu_p \stackrel{*}{\rightharpoonup} \nu_\infty,
\end{cases}
\]
as $p\to\infty$, along the sequence. Further, each \(u_p\) is a minimiser of the \(L^p\)-constrained problem
\[
\|f(D^ku_p)\|_{L^p(\Omega)}
=
\inf\Big\{
\|f(D^kv)\|_{L^p(\Omega)} :
v\in W_{\mathrm B}^{k,\alpha p}(\Omega;\mathbb{R}^N),
\ \|g(v,Dv,\ldots,D^{k-1}v)\|_{L^p(\Omega)}=1
\Big\},
\]
and the pair \((u_p,\Lambda_p)\) satisfies the corresponding weak Euler--Lagrange system.

Finally, the associated measures are given by
\begin{align}
\mathrm{M}_p
&=
\frac{1}{\mathcal{L}^n(\Omega)}
\left(\frac{f(D^ku_p)}{\Lambda_p}\right)^{p-1}
\partial f(D^ku_p)\,
\mathcal{L}^n\!\llcorner \Omega,
\\
\nu_p
&=
\frac{1}{\mathcal{L}^n(\Omega)}
g\big(u_p,Du_p,\ldots,D^{k-1}u_p\big)^{p-1}\,
\mathcal{L}^n\!\llcorner \Omega.
\label{1.12}
\end{align}
\end{theorem}

\section{Proof of the main result}

In this section we establish our main result, as a consequence of a series of lemmas and propositions.

The first lemma below verifies that there exists at least one function in the admissible set. It ensures that if you start with $v$ in the admissible class of the limiting problem, then $v$ remains in the admissible class after appropriate normalisation.
\begin{lemma}\label{lemma2}
Let \(v \in W_{\mathrm B}^{k,\infty}(\Omega;\mathbb R^N)\setminus\{0\}\). Then there exists a family
\[
(t_p)_{p\in (n/\alpha,\infty]}
\subseteq (0,\infty)
\]
such that \(t_p \to t_\infty\) as \(p\to\infty\), and
\[
\left\| g\bigl(t_p v,\, t_p Dv,\, \ldots,\, t_p D^{k-1}v\bigr)\right\|_{L^p(\Omega)}=1
\]
for every \(p\in (n/\alpha,\infty]\). Moreover, if
\[
\|g(v,Dv,\ldots,D^{k-1}v)\|_{L^\infty(\Omega)}=1,
\]
then \(t_\infty=1\).
\end{lemma}

\begin{proof}
Fix \(v \in W_{\mathrm B}^{k,\infty}(\Omega;\mathbb R^N)\setminus\{0\}\), and define
\[
\rho_\infty(t)
:=
\max_{x\in\overline{\Omega}}
g\bigl(tv(x),\, tDv(x),\, \ldots,\, tD^{k-1}v(x)\bigr),
\qquad t\ge 0.
\]
Since \(g\) is continuous and \(v,Dv,\ldots,D^{k-1}v\) are bounded and continuous on \(\overline{\Omega}\), the map
\(\rho_\infty\) is continuous on \([0,\infty)\). Also, because \(g(0,\ldots,0)=0\), we have
\[
\rho_\infty(0)=0.
\]

We first show that \(\rho_\infty\) is strictly increasing on \((0,\infty)\). Let
\[
(\eta,P_1,\ldots,P_{k-1})\in
\bigg(\mathbb R^N\times\prod_{j=1}^{k-1}\mathbb R^{N\times n^j}\bigg)\setminus\{0\},
\]
and define
\[
h(s):=g(s\eta,sP_1,\ldots,sP_{k-1}),\qquad s>0.
\]
By Assumption~\ref{g}\,\ref{partial},
\[
\begin{aligned}
h'(s)
&=
\partial_\eta g(s\eta,sP_1,\ldots,sP_{k-1})\cdot \eta
+\sum_{j=1}^{k-1}\partial_{P_j} g(s\eta,sP_1,\ldots,sP_{k-1}):P_j \\
&=
\frac{1}{s}
\bigg[
\partial_\eta g(s\eta,sP_1,\ldots,sP_{k-1})\cdot (s\eta)
+\sum_{j=1}^{k-1}\partial_{P_j} g(s\eta,sP_1,\ldots,sP_{k-1}):(sP_j)
\bigg] \\
&\ge \frac{C_7}{s}\, g(s\eta,sP_1,\ldots,sP_{k-1})
\\
&= \frac{C_7}{s}\, h(s).
\end{aligned}
\]
In particular, \(h'(s)>0\) whenever \(h(s)>0\), and therefore the function
\[
s\mapsto g(s\eta,sP_1,\ldots,sP_{k-1})
\]
is strictly increasing on \((0,\infty)\). It follows that if \(0\le s<t\), then for every \(x\in\overline{\Omega}\),
\[
g\bigl(sv(x),\, sDv(x),\, \ldots,\, sD^{k-1}v(x)\bigr)
\le
g\bigl(tv(x),\, tDv(x),\, \ldots,\, tD^{k-1}v(x)\bigr).
\]
Hence \(\rho_\infty(s)\le \rho_\infty(t)\), so \(\rho_\infty\) is nondecreasing. To prove strict monotonicity, fix \(t_0>0\), and let
\[
\Omega_{t_0}
:=
\Big\{
x\in\overline{\Omega}:
\rho_\infty(t_0)=g\bigl(t_0v(x),\, t_0Dv(x),\, \ldots,\, t_0D^{k-1}v(x)\bigr)
\Big\}.
\]
By Danskin's theorem, the right derivative of \(\rho_\infty\) at \(t_0\) exists and satisfies
\[
\rho_\infty'(t_0^+)
=
\max_{x\in \Omega_{t_0}}
\bigg[
\partial_\eta g_0(x)\cdot v(x)
+\sum_{j=1}^{k-1}\partial_{P_j} g_0(x):D^jv(x)
\bigg],
\]
where \(g_0(x)\) denotes \(g\bigl(t_0v(x),t_0Dv(x),\ldots,t_0D^{k-1}v(x)\bigr)\) and similarly for the partial derivatives. Rewriting this expression and using Assumption~\ref{g}\,\ref{partial}, we obtain
\[
\begin{aligned}
\rho_\infty'(t_0^+)
&=
\frac{1}{t_0}
\max_{x\in \Omega_{t_0}}
\bigg[
\partial_\eta g_0(x)\cdot (t_0v)
+\sum_{j=1}^{k-1}\partial_{P_j} g_0(x):(t_0D^jv)
\bigg] \\
&\ge
\frac{C_7}{t_0}
\max_{x\in\Omega_{t_0}}
g\bigl(t_0v(x),\, t_0Dv(x),\, \ldots,\, t_0D^{k-1}v(x)\bigr) \\
&=
\frac{C_7}{t_0}\,\rho_\infty(t_0).
\end{aligned}
\]
Since \(v\not\equiv 0\), there exists \(x_0\in\overline{\Omega}\) such that
\[
\bigl(v(x_0),Dv(x_0),\ldots,D^{k-1}v(x_0)\bigr)\neq 0,
\]
and thus 
\[
g(t_0v(x_0),t_0Dv(x_0),\ldots,t_0D^{k-1}v(x_0))>0, \text{ for } t_0>0. 
\]
Hence \(\rho_\infty(t_0)>0\), and therefore
\[
\rho_\infty'(t_0^+)>0.
\]
So \(\rho_\infty\) is strictly increasing on \((0,\infty)\). Next we show that \(\rho_\infty(t)\to\infty\) as \(t\to\infty\). Since \(v\not\equiv 0\), choose
\(\bar x\in\overline{\Omega}\) such that
\[
\bigl(v(\bar x),Dv(\bar x),\ldots,D^{k-1}v(\bar x)\bigr)\neq 0.
\]
By the coercivity assumption on \(g\), namely Assumption~\ref{g}\,\ref{coercive},
\[
g\bigl(tv(\bar x),\, tDv(\bar x),\, \ldots,\, tD^{k-1}v(\bar x)\bigr)\to\infty,
\]
as $t\to\infty$. Therefore,
\[
\rho_\infty(t)\ge g\bigl(tv(\bar x),\, tDv(\bar x),\, \ldots,\, tD^{k-1}v(\bar x)\bigr)\to\infty,
\]
as $t\to\infty$.  Since \(\rho_\infty\) is continuous, strictly increasing, satisfies \(\rho_\infty(0)=0\), and tends to \(+\infty\) as
\(t\to\infty\), there exists a unique \(t_\infty>0\) such that
\[
\rho_\infty(t_\infty)=1.
\]
Equivalently,
\[
\left\| g\bigl(t_\infty v,\, t_\infty Dv,\, \ldots,\, t_\infty D^{k-1}v\bigr)\right\|_{L^\infty(\Omega)}=1.
\]
If in addition
\[
\|g(v,Dv,\ldots,D^{k-1}v)\|_{L^\infty(\Omega)}=1,
\]
then by uniqueness we must have \(t_\infty=1\). Now fix \(p\in (n/\alpha,\infty)\), and define
\[
\rho_p(t)
:=
\fint_\Omega g\bigl(tv,\, tDv,\, \ldots,\, tD^{k-1}v\bigr)^p\, d\mathcal L^n,
\qquad t\ge 0.
\]
Again \(\rho_p\) is continuous and \(\rho_p(0)=0\). Moreover, since \(v\in W_{\mathrm B}^{k,\infty}(\Omega;\mathbb R^N)\),
Morrey's embedding gives
\[
v\in C^{k-1}(\overline{\Omega};\mathbb R^N).
\]
Because \(v\not\equiv 0\), the set
\[
E:=\Big\{x\in\Omega:\bigl(v(x),Dv(x),\ldots,D^{k-1}v(x)\bigr)\neq 0\Big\}
\]
has positive measure. For each \(x\in E\), the map
\[
t\mapsto g\bigl(tv(x),\, tDv(x),\, \ldots,\, tD^{k-1}v(x)\bigr)
\]
is increasing and tends to \(+\infty\) as \(t\to\infty\). Hence, by the monotone convergence theorem,
\[
\int_E g\bigl(tv,\, tDv,\, \ldots,\, tD^{k-1}v\bigr)^p\, d\mathcal L^n \to \infty,
\qquad \text{as } t\to\infty.
\]
Therefore \(\rho_p(t)\to\infty\), as \(t\to\infty\). By continuity, there exists \(t_p>0\) such that
\[
\rho_p(t_p)=1,
\]
that is,
\[
\left\| g\bigl(t_p v,\, t_p Dv,\, \ldots,\, t_p D^{k-1}v\bigr)\right\|_{L^p(\Omega)}=1.
\]
It remains to show that \(t_p\to t_\infty\) as \(p\to\infty\). Suppose for contradiction that this is false. Then there exist
a subsequence \((p_m)\) with \(p_m\to\infty\) and some
\[
t_0\in [0,\infty]\setminus\{t_\infty\}
\]
such that
\[
t_{p_m}\to t_0.
\]
We first exclude the possibility \(t_0=\infty\). If \(t_{p_m}\to\infty\), then after passing to a further subsequence we may assume that
\((t_{p_m})\) is increasing. Thus, for every \(x\in E\),
\[
g\bigl(t_{p_m}v(x),\, t_{p_m}Dv(x),\, \ldots,\, t_{p_m}D^{k-1}v(x)\bigr)\to\infty,
\]
and the sequence is pointwise increasing on \(E\). By the monotone convergence theorem,
\[
\int_E g\bigl(t_{p_m}v,\, t_{p_m}Dv,\, \ldots,\, t_{p_m}D^{k-1}v\bigr)^{p_m}\, d\mathcal L^n \to \infty.
\]
But this is impossible, because
\[
\fint_\Omega g\bigl(t_{p_m}v,\, t_{p_m}Dv,\, \ldots,\, t_{p_m}D^{k-1}v\bigr)^{p_m}\, d\mathcal L^n = 1
\]
for every \(m\). Hence \(t_0<\infty\). Since \(t_{p_m}\to t_0\) as $m\to\infty$ and \(v,Dv,\ldots,D^{k-1}v\) are continuous on \(\overline\Omega\), we have
\[
g\bigl(t_{p_m}v,\, t_{p_m}Dv,\, \ldots,\, t_{p_m}D^{k-1}v\bigr)
\to
g\bigl(t_0v,\, t_0Dv,\, \ldots,\, t_0D^{k-1}v\bigr)
\quad\text{uniformly on }\overline\Omega.
\]
Therefore,
\[
\left\| g\bigl(t_{p_m}v,\, t_{p_m}Dv,\, \ldots,\, t_{p_m}D^{k-1}v\bigr)\right\|_{L^{p_m}(\Omega)}
=
\left\| g\bigl(t_0v,\, t_0Dv,\, \ldots,\, t_0D^{k-1}v\bigr)\right\|_{L^{p_m}(\Omega)}
+o(1),
\]
as \(m\to\infty\). Since the left-hand side equals \(1\), passing to the limit and using the standard convergence of \(L^{p_m}\)-norms to the \(L^\infty\)-norm gives
\[
1
=
\left\| g\bigl(t_0v,\, t_0Dv,\, \ldots,\, t_0D^{k-1}v\bigr)\right\|_{L^\infty(\Omega)}
=
\rho_\infty(t_0).
\]
But \(\rho_\infty\) is strictly increasing and \(\rho_\infty(t_\infty)=1\), so necessarily \(t_0=t_\infty\), a contradiction. We therefore conclude that \(t_p\to t_\infty\) as \(p\to\infty\).
\end{proof}

The following lemma guarantees the existence of a minimiser $u_p$ for each finite-$p$ problem. It ensures that the functional is weakly lower semicontinuous, hence attaining its infimum.

\begin{lemma}\label{lemma3}
Let \(p>n/\alpha\). Then the \(L^p\)-constrained minimisation problem
\[
\|f(D^ku)\|_{L^p(\Omega)}
=
\inf\Bigl\{
\|f(D^kv)\|_{L^p(\Omega)} :
v\in W_{\mathrm B}^{k,\alpha p}(\Omega;\mathbb R^N),\ 
\|g(v,Dv,\ldots,D^{k-1}v)\|_{L^p(\Omega)}=1
\Bigr\}
\]
admits a minimiser
\[
u_p\in W_{\mathrm B}^{k,\alpha p}(\Omega;\mathbb R^N).
\]
\end{lemma}

\begin{proof}
Fix \(p\in (n/\alpha,\infty)\), and choose
\[
v_0\in W_{\mathrm B}^{k,\infty}(\Omega;\mathbb R^N)\setminus\{0\}.
\]
By Lemma~\ref{lemma2}, there exists \(t_p>0\) such that
\[
\big\|g\big(t_pv_0,t_pDv_0,\ldots,t_pD^{k-1}v_0\big) \big\|_{L^p(\Omega)}=1.
\]
Hence the admissible class is nonempty. Next we show that the functional
\[
\mathscr F_p(v):=\|f(D^kv)\|_{L^p(\Omega)}
\]
is weakly lower semicontinuous on \(W^{k,\alpha p}(\Omega;\mathbb R^N)\). Since \(f\) is
(Morrey) \(k\)-quasiconvex by Assumption~\ref{f}\,\ref{morrey}, Jensen's inequality implies that
\(f^p\) is also (Morrey) \(k\)-quasiconvex. Indeed, for every
\(X\in\mathbb R_s^{N\times n^k}\) and every \(\phi\in W_0^{k,\infty}(\Omega;\mathbb R^N)\),
\[
f(X)^p
\le
\left(\,\fint_\Omega f\bigl(X+D^k\phi(x)\bigr)\,d\mathcal L^n(x)\! \right)^p
\le
\fint_\Omega f\bigl(X+D^k\phi(x)\bigr)^p\,d\mathcal L^n(x).
\]
Moreover, by Assumption~\ref{f}\,\eqref{growth-a}, there exist constants \(C_5(p),C_6(p)>0\) such that
\[
f(X)^p\le C_5(p)|X|^{\alpha p}+C_6(p),
\qquad X\in\mathbb R_s^{N\times n^k}.
\]
Therefore, by \cite[Theorem~3.6]{zhou1998weak}, the map
\[
v\mapsto \int_\Omega f(D^kv)^p\,d\mathcal L^n
\]
is weakly lower semicontinuous on \(W^{k,\alpha p}(\Omega;\mathbb R^N)\), and hence so is
\(\mathscr F_p\) on the closed subspace \(W_{\mathrm B}^{k,\alpha p}(\Omega;\mathbb R^N)\).

Let now \((u_j)_{j\ge1}\subseteq W_{\mathrm B}^{k,\alpha p}(\Omega;\mathbb R^N)\) be a minimising sequence. Since the admissible class is nonempty and \(f\ge0\), the infimum is finite.

We now prove the coercivity of the functional \(\mathscr F_p\) on \(W_{\mathrm B}^{k,\alpha p}(\Omega;\mathbb R^N)\), considering the two cases of boundary conditions separately.

\medskip
\noindent\emph{Clamped case.}
Suppose first that \(u\in W_{\mathrm C}^{k,\alpha p}(\Omega;\mathbb R^N)=W_0^{k,\alpha p}(\Omega;\mathbb R^N)\). Then \(D^ju=0\) on \(\partial\Omega\) for \(j=0,\ldots,k-1\) in the trace sense. By Assumption~\ref{f}\,\eqref{growth-a},
\[
f(D^ku)\ge -C_3+C_4|D^ku|^\alpha
\qquad \text{a.e. on }\Omega.
\]
Then, successive applications of the Poincar\'e inequality to $u$ and its derivatives up to $(k-1)$-order, yields
\[
\|u\|_{W^{k,\alpha p}(\Omega)}
\le C\,\|D^ku\|_{L^{\alpha p}(\Omega)},
\]
where the constant $C=C(\Omega,k,\alpha,p)>0$ depends only on $\Omega,k,\alpha,p$. Consequently,
\[
\mathscr F_p(u)\ge c\,\|u\|_{W^{k,\alpha p}(\Omega)}^\alpha - C
\]
for suitable constants \(c,C>0\), independent of the mapping $u \in W_{\mathrm C}^{k,\alpha p}(\Omega;\mathbb R^N)$.

\medskip
\noindent\emph{Hinged case.}
Now let \(u\in W_{\mathrm H}^{k,\alpha p}(\Omega;\mathbb R^N)\). By definition,
\[
D^ju=0 \quad\text{on }\partial\Omega
\qquad\text{for }j=0,\ldots,k-2,
\]
in the trace sense. For \(j=0,\ldots,k-3\), since \(D^ju\) vanishes on \(\partial\Omega\), the standard Poincar\'e inequality gives
\[
\|D^ju\|_{L^{\alpha p}(\Omega)}\le C_j(\Omega, j,\alpha,p) \|D^{j+1}u\|_{L^{\alpha p}(\Omega)},
\]
for a constant depending only on $\Omega, j,\alpha,p$. Similarly,
\[
\|D^{k-2}u\|_{L^{\alpha p}(\Omega)}\le C_{k-2}(\Omega, k,\alpha,p) \|D^{k-1}u\|_{L^{\alpha p}(\Omega)}.
\]
For the highest order term, the Poincar\'e--Wirtinger inequality gives
\[
\left\| D^{k-1}u- \fint_\Omega D^{k-1}u\,d\mathcal L^n \right\|_{L^{\alpha p}(\Omega)}
\le
C_{k-1}(\Omega, k,\alpha,p) \|D^ku\|_{L^{\alpha p}(\Omega)},
\]
constant depending only on $\Omega, k,\alpha,p$. Since \(D^{k-2}u=0\) on \(\partial\Omega\), the Gauss--Green theorem implies
\[
\fint_\Omega D^{k-1}u\,d\mathcal L^n
=
\fint_{\partial\Omega} D^{k-2}u\otimes \nu\,d\mathcal H^{n-1}
=0.
\]
Hence
\[
\|D^{k-1}u\|_{L^{\alpha p}(\Omega)}
\le
C_{k-1}(\Omega, k,\alpha,p)\|D^ku\|_{L^{\alpha p}(\Omega)}.
\]

\noindent Combining all the estimates, we obtain (for a new constant)
\[
\sum_{j=0}^{k-1}\|D^ju\|_{L^{\alpha p}(\Omega)}
\le
C\,\|D^ku\|_{L^{\alpha p}(\Omega)},
\]
and therefore
\[
\|u\|_{W^{k,\alpha p}(\Omega)}
\le C'\,\|D^ku\|_{L^{\alpha p}(\Omega)}.
\]
Using Assumption~\ref{f}\,\eqref{growth-a} once again, it follows that
\[
\mathscr F_p(u)\ge c'\,\|u\|_{W^{k,\alpha p}(\Omega)}^\alpha - C'
\]
for suitable constants \(c',C'>0\).

In conclusion, in either boundary class, \(\mathscr F_p\) is coercive on
\(W_{\mathrm B}^{k,\alpha p}(\Omega;\mathbb R^N)\).

By reflexivity of \(W_{\mathrm B}^{k,\alpha p}(\Omega;\mathbb R^N)\), after passing to a subsequence we may assume that
\[
u_j \rightharpoonup u_p
\qquad\text{weakly in }W_{\mathrm B}^{k,\alpha p}(\Omega;\mathbb R^N).
\]
Since \(\alpha p>n\), Morrey's embedding yields
\[
u_j\to u_p \quad\text{in } C^{k-1}(\overline{\Omega};\mathbb R^N).
\]
In particular,
\[
(u_j,Du_j,\ldots,D^{k-1}u_j)\to (u_p,Du_p,\ldots,D^{k-1}u_p)
\quad\text{uniformly on }\overline{\Omega}.
\]
Because \(g\) is continuous, we deduce that
\[
g(u_j,Du_j,\ldots,D^{k-1}u_j)\to g(u_p,Du_p,\ldots,D^{k-1}u_p)
\quad\text{uniformly on }\overline{\Omega},
\]
and therefore the constraint passes to the limit:
\[
\|g(u_p,Du_p,\ldots,D^{k-1}u_p)\|_{L^p(\Omega)}=1.
\]

Finally, by weak lower semicontinuity,
\[
\mathscr F_p(u_p)
\le
\liminf_{j\to\infty} \mathscr F_p(u_j),
\]
so \(u_p\) attains the minimum. This completes the proof.
\end{proof}

This lemma derives the PDE system satisfied by the finite-$p$ minimisers $u_p$. The differentiability assumptions on $f$ and $g$ ensure that the functionals are G\^ateaux differentiable, allowing one to apply the Lagrange multiplier rule in Banach spaces and obtain the weak PDE form.

\begin{lemma}\label{lemma4}
Let \(p>n/\alpha\), and let \(u_p\) be the minimiser of the \(L^p\)-constrained
problem obtained in Lemma~\ref{lemma3}. Then there exists
\(\lambda_p\in\mathbb R\) such that the pair \((u_p,\lambda_p)\) satisfies
\begin{equation}\label{1.11}
\begin{aligned}
&\fint_\Omega
f(D^ku_p)^{p-1}\,\partial f(D^ku_p):D^k\phi
\,d\mathcal L^n \\
&\qquad=
\lambda_p
\fint_\Omega
g(u_p,Du_p,\ldots,D^{k-1}u_p)^{p-1}
\bigg(
\partial_\eta g_p\cdot\phi
+
\sum_{j=1}^{k-1}\partial_{P_j} g_p:D^j\phi
\bigg)
\,d\mathcal L^n ,
\end{aligned}
\end{equation}
for all test maps
\[
\phi\in W_{\mathrm B}^{k,\alpha p}(\Omega;\mathbb R^N),
\]
where we have written \(g_p := g(u_p,Du_p,\ldots,D^{k-1}u_p)\) and similarly for the partial derivatives.

Consequently \(u_p\) is a weak solution of the system
\[
D^k:\!\big(f(D^ku_p)^{p-1}\partial f(D^ku_p)\big)
=
\lambda_p
\sum_{j=0}^{k-1}(-1)^j D^j\!\cdot\!\big(
g_p^{p-1}\,\partial_{P_j} g_p
\big),
\]
where we use the convention \(\partial_{P_0}g = \partial_\eta g\), \(D^0\) is the identity, and \(D^j\cdot\) denotes the \(j\)-fold divergence.
\end{lemma}

In the hinged case there is an additional natural boundary condition
arising from the integration-by-parts procedure, since \(D^{k-1}u\) is free on
\(\partial\Omega\). We will not need this condition explicitly in the
sequel.

\begin{proof}
Define the functionals
\[
\mathcal F_p(u)
:=
\fint_\Omega f(D^ku)^p\,d\mathcal L^n,
\qquad
\mathcal G_p(u)
:=
\fint_\Omega g(u,Du,\ldots,D^{k-1}u)^p\,d\mathcal L^n .
\]
By Assumption~\ref{f}\,\eqref{growth-a}, the map
\(u\mapsto\mathcal F_p(u)\) is well defined on
\(W_{\mathrm B}^{k,\alpha p}(\Omega;\mathbb R^N)\).
Standard growth estimates imply that both \(\mathcal F_p\) and
\(\mathcal G_p\) are continuously Fr\'echet differentiable on this space. Since \(u_p\) minimises \(\mathcal F_p\) subject to the constraint
\[
\mathcal G_p(u)=1,
\]
the Lagrange multiplier rule in Banach spaces
(see e.g.\ \cite[p.~278]{zeidler2013nonlinear})
provides a scalar \(\lambda_p\in\mathbb R\) such that
\[
D\mathcal F_p(u_p)[\phi]
=
\lambda_p\,D\mathcal G_p(u_p)[\phi]
\qquad
\text{for all }\phi\in W_{\mathrm B}^{k,\alpha p}(\Omega;\mathbb R^N).
\]
Computing the first variations gives
\[
D\mathcal F_p(u_p)[\phi]
=
\fint_\Omega
f(D^ku_p)^{p-1}\partial f(D^ku_p):D^k\phi
\,d\mathcal L^n
\]
and
\[
D\mathcal G_p(u_p)[\phi]
=
\fint_\Omega
g_p^{p-1}
\bigg(
\partial_\eta g_p\cdot\phi
+
\sum_{j=1}^{k-1}\partial_{P_j} g_p:D^j\phi
\bigg)
\,d\mathcal L^n .
\]
Substituting these expressions into the multiplier identity yields the
weak formulation stated above. Finally, integrating by parts \(j\) times in each \(D^j\phi\)-term gives the divergence-form PDE.
\end{proof}

This lemma relates the Lagrange multipliers $\lambda_p$ to the $L^p$ norms of $f(D^ku_p)$, introducing $\Lambda_p = \lambda_p^{1/p}$. It establishes bounds for $\Lambda_p$, which are crucial for showing convergence of the eigenvalues as $p \to \infty$ and identifying the limiting $\Lambda_\infty$.

\begin{lemma}\label{lemma5}
Let \(\{(u_p,\lambda_p)\}_{p>n/\alpha}\) be the family obtained in
Lemma~\ref{lemma4}. Then for every \(p>n/\alpha\) there exists
\(\Lambda_p>0\) such that
\[
\lambda_p=(\Lambda_p)^p>0 .
\]
Define
\[
L_p:=\|f(D^ku_p)\|_{L^p(\Omega)} .
\]
Then the following bounds hold:
\[
0<
\left(\frac{C_1}{C_8}\right)^{1/p} L_p
\le
\Lambda_p
\le
\left(\frac{C_2}{C_7}\right)^{1/p} L_p .
\]
\end{lemma}
\begin{proof}
First observe that \(L_p>0\).
Indeed, the only map \(u\in W_B^{k,\alpha p}(\Omega;\mathbb R^N)\)
for which
\[
\|f(D^ku)\|_{L^p(\Omega)}=0
\]
is \(u\equiv0\), since \(f\ge0\) and \(f(X)=0\) only for \(X=0\).
However \(u\equiv0\) does not belong to the admissible class because
\[
\|g(u,Du,\ldots,D^{k-1}u)\|_{L^p(\Omega)}=0 \neq 1 .
\]
Hence \(L_p>0\). Next we use the Euler--Lagrange identity from Lemma~\ref{lemma4}.
Choosing the test function \(\phi=u_p\) gives
\[
\begin{aligned}
&\fint_\Omega
f(D^ku_p)^{p-1}\partial f(D^ku_p):D^ku_p
\,d\mathcal L^n \\
&\qquad=
\lambda_p
\fint_\Omega
g_p^{p-1}
\bigg(
\partial_\eta g_p\cdot u_p
+
\sum_{j=1}^{k-1}\partial_{P_j} g_p:D^ju_p
\bigg)
\,d\mathcal L^n .
\end{aligned}
\]
Using Assumption~\ref{f}\,\eqref{1.3c}, we obtain
\[
C_1
\fint_\Omega f(D^ku_p)^p
\le
\fint_\Omega
f(D^ku_p)^{p-1}\partial f(D^ku_p):D^ku_p
\le
C_2
\fint_\Omega f(D^ku_p)^p .
\]
Similarly, Assumption~\ref{g}\,\ref{partial} implies
\[
C_7 g(\eta,P_1,\ldots,P_{k-1})
\le
\partial_\eta g\cdot\eta
+\sum_{j=1}^{k-1}\partial_{P_j} g:P_j
\le
C_8 g(\eta,P_1,\ldots,P_{k-1}).
\]
Multiplying by \(g^{p-1}\) and integrating gives
\[
C_7
\fint_\Omega g_p^p
\le
\fint_\Omega
g_p^{p-1}
\bigg(\partial_\eta g_p\cdot u_p
+\sum_{j=1}^{k-1}\partial_{P_j} g_p:D^ju_p\bigg)
\le
C_8
\fint_\Omega g_p^p .
\]
Since the constraint implies
\[
\fint_\Omega g_p^p=1 ,
\]
combining the above inequalities with the Euler--Lagrange identity yields
\[
\frac{C_1}{C_8}L_p^p
\le
\lambda_p
\le
\frac{C_2}{C_7}L_p^p .
\]
Finally, define
\[
\Lambda_p:=\lambda_p^{1/p}>0 .
\]
Taking \(p\)-th roots, we arrive at the desired estimate.
\end{proof}

The next proposition shows that the sequence $(u_p)$ converges, up to a subsequence, to a limiting function $u_\infty$ that solves the $L^\infty$ minimisation problem. It bridges the finite-$p$ approximations and the limiting PDE, ensuring the convergence of both the minimisers and eigenvalues.

\begin{proposition}\label{proposition6}
There exists
\[
(u_\infty, \Lambda_\infty) \in \, W_{\mathrm B}^{k,\infty}(\Omega;\mathbb R^N) \times (0,\infty),
\]
and a sequence \((p_j)_{j\ge1}\subseteq (n/\alpha,\infty)\) with \(p_j\to\infty\), such that
\[
\begin{cases}
u_{p_j}\to u_\infty & \text{in } C^{k-1}(\overline\Omega;\mathbb R^N),\\[2mm]
D^ku_{p_j}\rightharpoonup^\ast D^ku_\infty & \text{in } L^\infty(\Omega;\mathbb R_s^{N\times n^k}),\\[2mm]
D^ku_{p_j}\rightharpoonup D^ku_\infty & \text{in } L^q(\Omega;\mathbb R_s^{N\times n^k}) \quad \text{for every } q\in (1,\infty),\\[2mm]
\Lambda_{p_j}\to \Lambda_\infty &
\end{cases}
\]
as $j\to\infty$. Moreover, \(u_\infty\) is a solution of the constrained \(L^\infty\) minimisation problem
\eqref{minimisation problem}, and
\[
\Lambda_\infty=\|f(D^ku_\infty)\|_{L^\infty(\Omega)}.
\]
\end{proposition}
\begin{proof}
Fix \(p>n/\alpha\), and let \(u_p\) be the minimiser obtained in Lemma~\ref{lemma3}.
Also fix a map
\[
v_0\in W_{\mathrm B}^{k,\infty}(\Omega;\mathbb R^N)\setminus\{0\}.
\]
By Lemma~\ref{lemma2}, there exists \(t_p>0\) such that
\[
\big\|g\big(t_pv_0,t_pDv_0,\ldots,t_pD^{k-1}v_0\big) \big\|_{L^p(\Omega)}=1,
\]
and \(t_p\to t_\infty\) as \(p\to\infty\). Hence \(t_pv_0\) is admissible for the
\(L^p\)-problem.

By minimality of \(u_p\),
\[
\|f(D^ku_p)\|_{L^p(\Omega)}
\le
\|f(t_pD^kv_0)\|_{L^p(\Omega)}.
\]
Since \(t_p\to t_\infty\) and \(v_0\in W_{\mathrm B}^{k,\infty}\), the right-hand side is uniformly bounded in \(p\). Therefore
\[
\sup_{p>n/\alpha}\|f(D^ku_p)\|_{L^p(\Omega)}<\infty.
\]
Using Assumption~\ref{f}\,\eqref{growth-a}, for every fixed \(q\in(1,\infty)\) and every \(p\ge q\) we obtain
\[
\sup_{p\ge q}\|D^ku_p\|_{L^{\alpha q}(\Omega)}<\infty.
\]
Arguing exactly as in the coercivity part of Lemma~\ref{lemma3}, the boundary conditions imply
\[
\sup_{p\ge q}\|u_p\|_{W^{k,\alpha q}(\Omega)}<\infty.
\]

Now fix \(q>n/\alpha\). Since \(W^{k,\alpha q}(\Omega;\mathbb R^N)\) is reflexive, after passing to a subsequence we may assume
\[
u_p \rightharpoonup u_\infty^{(q)}
\qquad\text{weakly in }W^{k,\alpha q}(\Omega;\mathbb R^N).
\]
Because \(\alpha q>n\), Morrey's embedding gives compactness into \(C^{k-1}(\overline\Omega;\mathbb R^N)\). A diagonal argument then provides
a map
\[
u_\infty \in \bigcap_{q<\infty} W_{\mathrm B}^{k,\alpha q}(\Omega;\mathbb R^N)
\]
and a sequence \(p_j\to\infty\) such that
\[
u_{p_j}\to u_\infty \quad \text{in } C^{k-1}(\overline\Omega;\mathbb R^N),
\]
and
\[
D^ku_{p_j}\rightharpoonup D^ku_\infty \quad \text{weakly in } L^{\alpha q}(\Omega;\mathbb R_s^{N\times n^k})
\]
for every finite \(q\). In particular,
\[
u_\infty\in W_{\mathrm B}^{k,\infty}(\Omega;\mathbb R^N).
\]

We next show that \(u_\infty\) is admissible. Since
\[
\big\|g\big(u_{p_j},Du_{p_j},\ldots,D^{k-1}u_{p_j}\big)\big\|_{L^{p_j}(\Omega)}=1
\]
for every \(j\), and since
\[
\big(u_{p_j},Du_{p_j},\ldots,D^{k-1}u_{p_j}\big)\to \big(u_\infty,Du_\infty,\ldots,D^{k-1}u_\infty \big)
\quad\text{uniformly on }\overline\Omega,
\]
continuity of \(g\) gives
\[
g\big(u_{p_j},Du_{p_j},\ldots,D^{k-1}u_{p_j}\big)
\to
g\big(u_\infty,Du_\infty,\ldots,D^{k-1}u_\infty\big)
\quad\text{uniformly on }\overline\Omega.
\]
Passing to the limit in the \(L^{p_j}\)-norm and using the standard convergence of \(L^{p_j}\)-norms to the \(L^\infty\)-norm, we obtain
\[
\big\| g\big(u_\infty,Du_\infty,\ldots,D^{k-1}u_\infty \big)\big\|_{L^\infty(\Omega)}=1.
\]
Thus \(u_\infty\) lies in the admissible class of \eqref{minimisation problem}.

Now let \(v\in W_{\mathrm B}^{k,\infty}(\Omega;\mathbb R^N)\) be any admissible competitor, i.e.
\[
\big\|g\big(v,Dv,\ldots,D^{k-1}v \big)\big\|_{L^\infty(\Omega)}=1.
\]
By Lemma~\ref{lemma2}, there exists \(t_p\to 1\) such that
\[
\big\|g\big(t_pv,t_pDv,\ldots,t_pD^{k-1}v\big)\big\|_{L^p(\Omega)}=1
\]
for all \(p>n/\alpha\). By minimality of \(u_p\),
\[
\|f(D^ku_p)\|_{L^p(\Omega)}
\le
\|f(t_pD^kv)\|_{L^p(\Omega)}.
\]
Fix any \(q\in(1,\infty)\). By H\"older's inequality and the above,
\[
\|f(D^ku_{p_j})\|_{L^q(\Omega)}
\le
\|f(D^ku_{p_j})\|_{L^{p_j}(\Omega)}
\le
\|f(t_{p_j}D^kv)\|_{L^{p_j}(\Omega)}.
\]

By the weak lower semicontinuity of the functional on \(W_{\mathrm B}^{k,\alpha q}(\Omega;\mathbb R^N)\), we may let \(p_j\to\infty\) to obtain
\[
\|f(D^ku_\infty)\|_{L^q(\Omega)}
\le
\liminf_{p_j\to\infty} \|f(D^ku_{p_j})\|_{L^{p_j}(\Omega)}
\le
\|f(D^kv)\|_{L^\infty(\Omega)}.
\]
Letting \(q\to\infty\) gives
\[
\|f(D^ku_\infty)\|_{L^\infty(\Omega)}
\le
\|f(D^kv)\|_{L^\infty(\Omega)}. 
\]

Since \(v\) was arbitrary, \(u_\infty\) minimises \eqref{minimisation problem}. We now define
\[
\Lambda_\infty:=\|f(D^ku_\infty)\|_{L^\infty(\Omega)}.
\]
Clearly \(\Lambda_\infty\ge 0\). In fact \(\Lambda_\infty>0\), because the only map in
\(W_{\mathrm B}^{k,\infty}(\Omega;\mathbb R^N)\) satisfying \(f(D^ku)\equiv 0\) is the trivial map \(u\equiv 0\), while
\(u\equiv 0\) is not admissible since \(g(0,\ldots,0)=0\). Finally, by Lemma~\ref{lemma5},
\[
\left(\frac{C_1}{C_8}\right)^{1/p_j}L_{p_j}
\le
\Lambda_{p_j}
\le
\left(\frac{C_2}{C_7}\right)^{1/p_j}L_{p_j},
\]
where
\[
L_{p_j}:=\|f(D^ku_{p_j})\|_{L^{p_j}(\Omega)}.
\]

Since \(L_{p_j}\to \|f(D^ku_\infty)\|_{L^\infty(\Omega)}=\Lambda_\infty\), we conclude that
\[
\Lambda_{p_j}\to \Lambda_\infty,
\]
as $j\to\infty$. This completes the proof.
\end{proof}

\begin{remark}
The derivation of an explicit geometric upper bound for \(\Lambda_\infty\) in the \(k\)-th order case requires a genuine \(k\)-th order adaptation of the boundary test-function construction used in the second-order setting of \cite{clark2024generalized}. In particular, one needs estimates for \(D^k(d_\Omega^m\zeta)\) for an appropriate power \(m\), which in turn require suitable \(C^k\)-regularity and tubular-neighbourhood control for the distance function \(d_\Omega\). Since this analysis grows substantially more involved with the order \(k\), and we do not have an immediate application of this explicit bound, we refrain from including the relevant estimate in this work.
\end{remark}

We now establish the lower bound on the eigenvalues.

\begin{proposition}[Lower bound for the eigenvalue]\label{lower bound}
Under the hypotheses of Theorem~\ref{main theorem}, the eigenvalue
\(\Lambda_\infty=\|f(D^ku_\infty)\|_{L^\infty(\Omega)}\) satisfies
\[
\Lambda_\infty \geq \left( C_4  \Bigg\{\displaystyle\sum_{j=0}^{k-1}
 \bigg( \prod_{l=j}^{k-1} C_{l}(\Omega) \bigg)\|\partial_{P_j} g\|_{L^\infty(\{g\le1\})} \Bigg\}^{\!\!-\alpha}
-C_3
\right)^+,
\]
where \(\partial_{P_0}g \equiv \partial_\eta g\), and \((\cdot)^+\) denotes the positive part. 

In the above estimate, \(C_0(\Omega),\ldots,C_{k-2}(\Omega)\) denote constants arising from the associated Poincar\'e inequality applied to the derivatives of corresponding order, and $C_{k-1}(\Omega)$ denotes the constant arising from either the Poincar\'e, or the Poincar\'e-Wirtinger inequalities applied to the derivatives of $(k-1)$-order, depending on the type of boundary data (clamped/hinged).
\end{proposition}

\begin{proof}
Since \(u_\infty\in W_{\mathrm B}^{k,\infty}(\Omega;\mathbb R^N)\) and
\(\alpha p>n\) for all \(p\) in the approximating sequence, Morrey's
embedding theorem gives
\(u_\infty\in C^{k-1}(\overline\Omega;\mathbb R^N)\). SInce \(g\) is
continuous and \(\overline\Omega\) is compact, the function
\[
x\mapsto g\bigl(u_\infty(x),Du_\infty(x),\ldots,D^{k-1}u_\infty(x)\bigr)
\]
attains its supremum on \(\overline\Omega\). The admissibility constraint
\[
\big\|g\big(u_\infty,Du_\infty,\ldots,D^{k-1}u_\infty\big)
\big\|_{L^\infty(\Omega)}=1
\]
therefore provides a point \(x_0\in\overline\Omega\) at which
\begin{equation}\label{lb:max}
g\bigl(u_\infty(x_0),Du_\infty(x_0),\ldots,D^{k-1}u_\infty(x_0)\bigr)=1,
\end{equation}
which is possible due to continuous differentiability of $u_\infty$ up to order $k-1$. For brevity we set
\[
(\eta^0,P_1^0,\ldots,P_{k-1}^0)
:=
\bigl(u_\infty(x_0),Du_\infty(x_0),\ldots,D^{k-1}u_\infty(x_0)\bigr).
\]
Since \(g(0,\ldots,0)=0\) and \(g\in C^1\), the fundamental theorem of
calculus along the line segment from the origin gives
\[
g(\eta^0,P_1^0,\ldots,P_{k-1}^0)
=
\int_0^1
\frac{d}{dt}\,\Big(
g\bigl(t\eta^0,\,tP_1^0,\ldots,tP_{k-1}^0\bigr)\Big)\,dt.
\]
Computing the derivative in \(t\) yields
\begin{equation}\label{lb:ftc}
g(\eta^0,P_1^0,\ldots,P_{k-1}^0)
=
\int_0^1
\bigg(
\partial_\eta g_t\cdot\eta^0
+
\sum_{j=1}^{k-1}\partial_{P_j} g_t:P_j^0
\bigg)\,dt,
\end{equation}
where we abbreviate \(g_t:=g(t\eta^0,tP_1^0,\ldots,tP_{k-1}^0)\), and likewise for its
partial derivatives.

We claim that for every \(t\in[0,1]\) the point \((t\eta^0,tP_1^0,\ldots,tP_{k-1}^0)\) lies in the sublevel set \(\{g\le1\}\). Indeed, as shown in the proof of Lemma~\ref{lemma2}, the
map
\[
s\mapsto g\big(s\eta^0,sP_1^0,\ldots,sP_{k-1}^0\big)
\]
is nondecreasing on \([0,\infty)\)
(this follows from Assumption~\ref{g}\,\ref{partial}). Hence, for
\(0\le t\le 1\),
\[
g\big(t\eta^0,tP_1^0,\ldots,tP_{k-1}^0\big)
\le
g\big(\eta^0,P_1^0,\ldots,P_{k-1}^0\big)=1.
\]
Consequently, the integrand in \eqref{lb:ftc} can be estimated by
\[
\bigg|
\partial_\eta g_t\cdot\eta^0
+
\sum_{j=1}^{k-1}\partial_{P_j}g_t:P_j^0
\bigg|
\le
\|\partial_\eta g\|_{L^\infty(\{g\le1\})}|\eta^0|
+
\sum_{j=1}^{k-1}
\|\partial_{P_j} g\|_{L^\infty(\{g\le1\})}|P_j^0|,
\]
for $t\in [0,1]$. Integrating over \(t\in[0,1]\) and using \eqref{lb:max}, we obtain the estimate
\begin{equation}
\label{lb:bound-g}
\begin{split}
1
& =g(\eta^0,P_1^0,\ldots,P_{k-1}^0)
\\
&=
\bigg| \int_0^1
\bigg(
\partial_\eta g_t\cdot\eta^0
+
\sum_{j=1}^{k-1}\partial_{P_j} g_t:P_j^0
\bigg)\,dt \bigg|
\\
&\leq \max_{t\in [0,1]} \bigg|
\partial_\eta g_t\cdot\eta^0
+
\sum_{j=1}^{k-1}\partial_{P_j} g_t:P_j^0
\bigg|
\\
&\leq
\sum_{j=0}^{k-1}
\|\partial_{P_j} g\|_{L^\infty(\{g\le1\})}\,
\|D^ju_\infty\|_{L^\infty(\Omega)},
\end{split}
\end{equation}
where we have symbolised  \(\partial_{P_0}g \equiv \partial_\eta g\), and used the
fact that \(|D^ju_\infty(x_0)|\le\|D^ju_\infty\|_{L^\infty(\Omega)}\) for $j=0,1,\ldots,k-1$ (recall that $u_\infty \in C^{k-1}(\overline{\Omega};\mathbb R^N)$).

We now bound the lower-order norms in terms of
\(\|D^ku_\infty\|_{L^\infty(\Omega)}\). In either boundary class, we have \(D^ju_\infty=0\) on \(\partial\Omega\) for \(j=0,\ldots,k-2\). Therefore, by the Poincar\'e inequality applied
successively, there exist constants \(C_0(\Omega),\ldots,C_{k-2}(\Omega)>0\)
such that
\begin{equation}\label{lb:poincare-lower}
\|D^ju_\infty\|_{L^\infty(\Omega)}
\le
C_j(\Omega)\,\|D^{j+1}u_\infty\|_{L^\infty(\Omega)},
\qquad j=0,\ldots,k-2.
\end{equation}
For the term \(D^{k-1}u_\infty\), we distinguish the two boundary
classes.

\medskip
\noindent\emph{\underline{Clamped case}.}
In this case \(D^{k-1}u_\infty=0\) on \(\partial\Omega\). By the Poincar\'e inequality, there exists $C_{k-1}(\Omega)>0$ such that
\[
|D^{k-1}u_\infty(x)|
\le
\|D^{k-1} u_\infty\|_{L^\infty(\Omega)}
\le
C_{k-1}(\Omega)\,\|D^ku_\infty\|_{L^\infty(\Omega)}.
\]

\noindent\emph{\underline{Hinged case}.}
Here \(D^{k-2}u_\infty=0\) on \(\partial\Omega\), so by the
Gauss--Green theorem (as in the proof of Lemma~\ref{lemma3}),
\[
\fint_\Omega D^{k-1}u_\infty\,d\mathcal L^n
=
\fint_{\partial\Omega} D^{k-2}u_\infty\otimes\nu\,d\mathcal H^{n-1}
=0.
\]
Hence, by the Poincar\'e-Wirtinger inequality, there exists $C_{k-1}(\Omega)>0$ such that
\[
\begin{split}
\big| D^{k-1}u_\infty(x) \big| &\leq \big\|D^{k-1} u_\infty \big\|_{L^\infty(\Omega)}
\\
&= \bigg\| D^{k-1}u - \fint_\Omega D^{k-1}u_\infty\,d\mathcal L^n\bigg\|_{L^\infty(\Omega)}
\\
& \le
C_{k-1}(\Omega)\,\|D^ku_\infty\|_{L^\infty(\Omega)}.
\end{split}
\]
In both cases we therefore have
\begin{equation}\label{lb:top-order}
\|D^{k-1}u_\infty\|_{L^\infty(\Omega)}
\le
C_{k-1}(\Omega) \,\|D^ku_\infty\|_{L^\infty(\Omega)},
\end{equation}
where $C_{k-1}(\Omega)$ is the constant arising either from the Poincar\'e, or the Poicar\'e-Wirtinger inequality, depending on the type of boundary condition (clamped/hinged). By iterating  \eqref{lb:poincare-lower} and \eqref{lb:top-order}, we obtain the estimate
\begin{equation}
\label{bound}
\begin{split}
\|D^j u_\infty\|_{L^\infty(\Omega)} &\leq C_{j}(\Omega) \,\|D^{j+1}u_\infty\|_{L^\infty(\Omega)}
\\
&\leq C_{j}(\Omega)  C_{j+1}(\Omega) \,\|D^{j+2}u_\infty\|_{L^\infty(\Omega)}
\\
& \ \  \vdots
\\
&\leq \bigg( \prod_{l=j}^{k-1} C_{l}(\Omega) \bigg) \,\|D^k u_\infty\|_{L^\infty(\Omega)}.
\end{split}
\end{equation}
Substituting \eqref{bound} into \eqref{lb:bound-g} yields
\begin{equation}
\label{lb:Dk-lower}
\begin{split}
1 &\leq
\sum_{j=0}^{k-1}
\|\partial_{P_j} g\|_{L^\infty(\{g\le1\})}\,
\|D^ju_\infty\|_{L^\infty(\Omega)}
\\
&\leq \Bigg\{ \sum_{j=0}^{k-1}
\|\partial_{P_j} g\|_{L^\infty(\{g\le1\})} \bigg( \prod_{l=j}^{k-1} C_{l}(\Omega) \bigg) \Bigg\} \|D^k u_\infty\|_{L^\infty(\Omega)}
\end{split}
\end{equation}
Finally, by Assumption~\ref{f}\,\eqref{growth-a}, for a.e.\
\(x\in\Omega\),
\[
f\bigl(D^ku_\infty(x)\bigr)
\ge
-C_3+C_4\,|D^ku_\infty(x)|^\alpha.
\]
Taking the essential supremum over \(x\in\Omega\) gives
\[
\Lambda_\infty
=
\|f(D^ku_\infty)\|_{L^\infty(\Omega)}
\ge
-C_3+C_4\,\|D^ku_\infty\|_{L^\infty(\Omega)}^\alpha.
\]
Combining with \eqref{lb:Dk-lower} and taking the positive part, we
conclude that
\[
\Lambda_\infty
\ge
\left( C_4  \Bigg\{\displaystyle\sum_{j=0}^{k-1}
 \bigg( \prod_{l=j}^{k-1} C_{l}(\Omega) \bigg)\|\partial_{P_j} g\|_{L^\infty(\{g\le1\})} \Bigg\}^{\!\!-\alpha}
-C_3
\right)^+,
\]
as claimed.
\end{proof}

The following lemma establishes the compactness and convergence of the measures $\nu_p$ and $\mathbf{M}_p$ arising from the $L^p$ approximations. The limits $\nu_\infty$ and $\mathbf{M}_\infty$ are precisely those appearing in the limiting PDE system of Theorem~\ref{main theorem}(B). It ensures weak* compactness in the space of Radon measures, which is essential for passing to the limit in the Euler--Lagrange equations.

\begin{lemma}\label{lemma7}
Let \(p>(n/\alpha)+2\). Then, along a subsequence \((p_j)_{j\ge1}\) with \(p_j\to\infty\), there exist measures
\[
\nu_\infty \in \mathcal M(\overline{\Omega}),
\qquad
\mathbf M_\infty \in \mathcal M\bigl(\overline{\Omega};\mathbb R_s^{N\times n^k}\bigr),
\]
such that
\[
\begin{cases}
\nu_{p_j}\stackrel{*}{\rightharpoonup}\nu_\infty
& \text{in }\mathcal M(\overline{\Omega}),\\[2mm]
\mathbf M_{p_j}\stackrel{*}{\rightharpoonup}\mathbf M_\infty
& \text{in }\mathcal M\bigl(\overline{\Omega};\mathbb R_s^{N\times n^k}\bigr),
\end{cases}
\]
as $j\to\infty$, where the approximating measures \(\nu_p\) and \(\mathbf M_p\) are defined by \eqref{1.12}.
\end{lemma}

\begin{proof}
We first establish a uniform bound for \(\nu_p\). Since \(g\ge0\) and
\[
\big\|g\big(u_p,Du_p,\ldots,D^{k-1}u_p\big)\big\|_{L^p(\Omega)}=1,
\]
it follows from \eqref{1.12} that
\[
\|\nu_p\|(\overline{\Omega})
=
\nu_p(\overline{\Omega})
=
\fint_\Omega g\big(u_p,Du_p,\ldots,D^{k-1}u_p \big)^{p-1}\,d\mathcal L^n.
\]
By H\"older's inequality,
\[
\fint_\Omega g\big(u_p,Du_p,\ldots,D^{k-1}u_p\big)^{p-1}\,d\mathcal L^n
\le
\left(\,
\fint_\Omega g\big(u_p,Du_p,\ldots,D^{k-1}u_p\big)^p\,d\mathcal L^n
\right)^{\frac{p-1}{p}}
=1.
\]
Hence
\[
\sup_{p>(n/\alpha)+2}\|\nu_p\|(\overline{\Omega})\le 1.
\]
By weak\(^*\) sequential compactness in \(\mathcal M(\overline{\Omega})\), there exists a subsequence, still denoted by \((p_j)\), and a measure
\[
\nu_\infty\in \mathcal M(\overline{\Omega})
\]
such that
\[
\nu_{p_j}\stackrel{*}{\rightharpoonup}\nu_\infty
\qquad\text{in }\mathcal M(\overline{\Omega}),
\]

as $j\to\infty$. We next prove a uniform bound for \(\mathbf M_p\). By definition,
\[
\mathbf M_p
=
\frac{1}{\mathcal L^n(\Omega)}
\left(\frac{f(D^ku_p)}{\Lambda_p}\right)^{p-1}
\partial f(D^ku_p)\,\mathcal L^n\!\llcorner\Omega.
\]
Therefore,
\[
\|\mathbf M_p\|(\overline{\Omega})
=
\fint_\Omega
\left(\frac{f(D^ku_p)}{\Lambda_p}\right)^{p-1}
|\partial f(D^ku_p)|\,d\mathcal L^n.
\]
Using Assumption~\ref{f}\,\eqref{growth-b}, we have
\[
|\partial f(X)|\le C_5 f(X)^\beta + C_6,
\qquad X\in\mathbb R_s^{N\times n^k}.
\]
Hence
\[
\begin{aligned}
\|\mathbf M_p\|(\overline{\Omega})
&\le
\frac{C_5}{\Lambda_p^{p-1}}
\fint_\Omega f(D^ku_p)^{p-1+\beta}\,d\mathcal L^n
+
\frac{C_6}{\Lambda_p^{p-1}}
\fint_\Omega f(D^ku_p)^{p-1}\,d\mathcal L^n.
\end{aligned}
\]
Applying H\"older's inequality to each term gives
\[
\|\mathbf M_p\|(\overline{\Omega})
\le
\left(\frac{L_p}{\Lambda_p}\right)^{p-1}
\bigl(C_5L_p^\beta + C_6\bigr).
\]
By Lemma~\ref{lemma5},
\[
\left(\frac{L_p}{\Lambda_p}\right)^{p-1}
\le
\left(\frac{C_8}{C_1}\right)^{1-\frac1p}.
\]
Moreover, along the subsequence under consideration, \(\Lambda_p\to\Lambda_\infty\), and again by Lemma~\ref{lemma5}, \(L_p\) is bounded. Consequently,
\[
\|\mathbf M_p\|(\overline{\Omega})
\le
\left(\frac{C_8}{C_1}\right)^{1-\frac1p}
\bigl(C_5(\Lambda_\infty+1)^\beta + C_6\bigr),
\]
for all sufficiently large \(p\). Thus \((\mathbf M_p)\) is uniformly bounded in
\(\mathcal M(\overline{\Omega};\mathbb R_s^{N\times n^k})\). By weak\(^*\) sequential compactness once more, after perhaps passing to a further subsequence, there exists
\[
\mathbf M_\infty\in \mathcal M\bigl(\overline{\Omega};\mathbb R_s^{N\times n^k}\bigr)
\]
such that
\[
\mathbf M_{p_j}\stackrel{*}{\rightharpoonup}\mathbf M_\infty
\qquad\text{in }\mathcal M\bigl(\overline{\Omega};\mathbb R_s^{N\times n^k}\bigr),
\]
as $j\to\infty$. This proves the claim.
\end{proof}

Lemma~\ref{lemma8} that follows confirms that the limiting quadruple
\[
(u_\infty, \Lambda_\infty, \mathbf{M}_\infty, \nu_\infty)
\]
satisfies the weak formulation of the PDE system
\begin{align*}
\int_{\overline{\Omega}} D^k\varphi : d\mathbf{M}_\infty &= \Lambda_\infty \int_{\overline{\Omega}} \bigg( \partial_\eta g(u_\infty,Du_\infty,\ldots,D^{k-1}u_\infty)\cdot\varphi \\
&\qquad\qquad + \sum_{j=1}^{k-1}\partial_{P_j} g(u_\infty,Du_\infty,\ldots,D^{k-1}u_\infty):D^j\varphi \bigg)\, d\nu_\infty.
\end{align*}
It establishes that the limiting function $u_\infty$ solves the Euler--Lagrange equations in a measure-valued sense, completing the proof of Theorem~\ref{main theorem}.

\begin{lemma}\label{lemma8}
Let $\mathbf M_\infty \in \mathcal M(\overline\Omega;\mathbb R_s^{N\times n^k})$
and $\nu_\infty\in\mathcal M(\overline\Omega)$ be the measures obtained in
Lemma~\ref{lemma7}. Then the quadruple
\[
(u_\infty,\Lambda_\infty,\mathbf M_\infty,\nu_\infty)
\]
satisfies the limiting weak system
\begin{align*}
\int_{\overline{\Omega}} D^k\phi : d\mathbf M_\infty
&=
\Lambda_\infty
\int_{\overline{\Omega}}
\bigg(
\partial_\eta g_\infty\cdot\phi
+
\sum_{j=1}^{k-1}\partial_{P_j} g_\infty:D^j\phi
\bigg)\,d\nu_\infty ,
\end{align*}
for all test maps $\phi\in C_B^k(\overline\Omega;\mathbb R^N)$, where \(g_\infty:=g(u_\infty,Du_\infty,\ldots,D^{k-1}u_\infty)\).
\end{lemma}
\begin{proof}
Fix a test map $\phi\in C_B^k(\overline\Omega;\mathbb R^N)$ and let
$p>(n/\alpha)+2$. By the definition of the measures $\mathbf M_p$ and $\nu_p$ in
\eqref{1.12}, the Euler--Lagrange system \eqref{1.11} can be written in the
measure form
\begin{align*}
\int_{\overline{\Omega}} D^k\phi : d\mathbf M_p
&=
\Lambda_p
\int_{\overline{\Omega}}
\bigg(
\partial_\eta g_p\cdot\phi
+
\sum_{j=1}^{k-1}\partial_{P_j} g_p:D^j\phi
\bigg)\,d\nu_p .
\end{align*}

By Proposition~\ref{proposition6}, along the subsequence $(p_j)$ we have
\[
\Lambda_{p_j}\to\Lambda_\infty,
\qquad
\big(u_{p_j},Du_{p_j},\ldots,D^{k-1}u_{p_j}\big)
\to
\big(u_\infty,Du_\infty,\ldots,D^{k-1}u_\infty\big)
\]
uniformly on $\overline\Omega$, as $j\to\infty$. Since $g\in C^1$, it follows that
\[
\partial_\eta g_{p_j}
\to
\partial_\eta g_\infty,
\qquad
\partial_{P_j} g_{p_j}
\to
\partial_{P_j} g_\infty
\quad (j=1,\ldots,k-1),
\]
uniformly on $\overline\Omega$, as $j\to\infty$. Moreover, by Lemma~\ref{lemma7},
\[
\nu_{p_j}\stackrel{*}{\rightharpoonup}\nu_\infty,
\qquad
\mathbf M_{p_j}\stackrel{*}{\rightharpoonup}\mathbf M_\infty
\]
in the sense of Radon measures. Passing to the limit in the weak formulation above and using the weak*--strong continuity of the duality pairing
\[
\mathcal M(\overline\Omega)\times C(\overline\Omega)\to\mathbb R,
\]
we obtain the claimed identity. This proves the claim.
\end{proof}

\subsubsection*{Declarations.} The authors have no conflict of interest to declare. There are no data in any way associated with this work. No AI assistance has been utilised in the preparation of this work.


\bibliography{references}
\bibliographystyle{abbrvnat}

\end{document}